\documentclass[12pt,reqno]{amsart}
\usepackage{amsmath}
\usepackage{}
\usepackage{}
\usepackage{amsfonts}
\usepackage{a4wide}
\usepackage{amsmath,amsthm,amssymb,amscd}
\usepackage{latexsym}
\usepackage{hyperref}
\usepackage[numbers,sort&compress]{natbib}
\usepackage{hypernat}
\allowdisplaybreaks
\numberwithin{equation}{section}

\newtheorem{theorem}{Theorem}[section]
\newtheorem{proposition}[theorem]{Proposition}

\newtheorem{lemma}[theorem]{Lemma}

\theoremstyle{definition}

\newcommand{\ds}{\displaystyle}

\def\R{\mathbb R}
\def\N{\mathbb N}

\usepackage{color}

\begin{document}

\title[Fractional Schr\"{o}dinger Equations]
{Semi-classical analysis for Fractional Schr\"{o}dinger Equations with fast decaying potentials}

 \author{Xiaoming An,\,\,\, Lipeng Duan* \,\,\,and \,\,\,Yanfang  Peng }

 \address{School of Mathematics and Statistics\, \&\, Guizhou University of Finance and Economics, Guiyang, 550025, P. R. China}
\email{651547603@qq.com}

\address{School of Mathematics and Statistics\, \&\, Hubei Key Laboratory of Mathematical Sciences, Central China
Normal University, Wuhan, 430079, P. R. China }

\email{ahudlp@sina.com}

\address{School of mathematical sciences\, \&\,  Guizhou Normal University, Guiyang, 550025, P. R. China}

\email{pyfang2005@sina.com}

\begin{abstract}
We study the following fractional Schr\"{o}dinger equation
\begin{equation*}\label{eq0.1}
\epsilon^{2s}(-\Delta)^s u + V(x)u = |u|^{p - 2}u, \,\,x\in\,\,\mathbb{R}^N,
\end{equation*}
 where $s\in (0,\,1)$, $N>2s$, $p>1$ is subcritical and $V(x)$ is a nonnegative continuous potential. We use penalized technique and variational methods to show that the problem has a family of solutions concentrating at a local minimum of $V(x)$ as $\epsilon\to 0$ provided that $\frac{2s}{N-2s}+2<p<\frac{2N}{N-2s}$. The novelty is that $V$ can decay arbitrarily or even be compactly supported.

\end{abstract}

\maketitle

\section{Introduction and main results}

In this paper, we consider the fractional Schr\"{o}dinger equation
\begin{equation}\label{eq1.1}
\epsilon^{2s}(-\Delta)^s u + V(x)u = |u|^{p - 2}u ,\,\,x\in\,\,\mathbb{R}^N,
\end{equation}
where $N > 2s$, $s\in (0,1)$, $\epsilon > 0$ is a small parameter, $p\in (2, 2^*_s)$, $2^*_s=2N/(N-2s)$, $V(x)\in C(\R^N,[0,+\infty))$ is a continuous nonnegative potential. Problem~\eqref{eq1.1}  arise from the study of time-independent waves $\psi(x,t) = e^{-iEt/\epsilon}u(x)$ of the following nonlinear fractional Schr\"{o}dinger equation
\begin{equation}\label{eqNLFS}
i\epsilon\frac{\partial{\psi}}{{\partial t}} = \epsilon^{2s}(-\Delta)^s{\psi} + U(x)\psi - f(\psi)\ x\in \mathbb{R}^N.
\end{equation}
Equation  \eqref{eq1.1} has a wide application in Physics, for example, the Einstein's theory of relativity, phase transition, conservation laws and fractional quantum mechanic, see \cite{2,3,4} and the references therein for more details. The parameter $\epsilon$ is the adimensionalised Planck constant; when $\epsilon$ goes to $0$, the quantum physics will shift to the classical physics, in which case the analysis for \eqref{eq1.1} is called semiclassical analysis. In the local case $s=1$, the existence and concentration of solutions for equation \eqref{eq1.1} have been studied extensively, see  \cite{{5},{6},{7},{8},{9},{10},{12},{13},{14},{15}} and their references therein for example.

In the nonlocal case $0<s<1$, to our best knowledge, there are few results on existence and concentration phenomena for \eqref{eq1.1}. Actually,  due to  the nonlocal effect of $(-\Delta)^s$,  the techniques created for local case can not be applied easily, see \cite{19,20} for example. Considering
the non-vanishing case, i.e., $\inf_{x\in \mathbb{R}^N}V(x) > 0$, Alves et al. in \cite{100} used the penalized method developed by Del Pino et al  in \cite{10} and extension technical  developed by Caffarelli et al. in \cite{24} to construct a family of positive solutions that concentrate at a local minimum of $V$ as $\epsilon\to 0$. Also in the non-vanishing case, references \cite{{25},{26},{27}} proved that the concentration phenomena occurs at the nondegenerate critical points of $V$ by using a reduction method. For the other results of \eqref{eq1.1} with non-vanishing potentials, we would like to refer the readers to \cite{19,20,J.Davila-M.del Pino-J.Wei-2014,V.Ambrosio-2019}.

Throughout all the works before, an assumption that $\inf_{\R^N}V>0$ was required. The study of the case that $V>0$ but $\inf_{\R^N}V=0\, ( \liminf_{|x|\to\infty}V(x)=0$) was initiated by X.  An et al. in \cite{APX}, where it was proved by constructing a new penalized function  that equation \eqref{eq1.1} has a family of positive solutions concentrating at a positive local minimum of $V$ if $\liminf_{|x|\to\infty}V(x)|x|^{2s}>0$.

In this paper, we try to solve the question of existence and concentration for more general potentials including fast decaying potentials, i.e. potentials for which
$$
\liminf_{|x|\to\infty}V(x)|x|^{2s}=0.
$$
A special case of a fast decaying potential is a potential $V$ that vanishes identically, i.e., $V$ is compactly supported.

In order to state our main result, we need to give some notations and assumptions. For  $s\in(0,1)$, the fractional Sobolev space $H^s(\mathbb{R}^N)$ is defined as
$$
H^s(\mathbb{R}^N) = \Big\{u\in L^2(\mathbb{R}^N):\frac{u(x) - u(y)}{|x - y|^{N/2 + s}}\in L^2(\mathbb{R}^N\times\mathbb{R}^N)\Big\},
$$
endowed with the norm
$$
\|u\|_{H^s(\mathbb{R}^N)} = \Big(\int_{\mathbb{R}^N}|(-\Delta)^{s/2}u|^2 + u^2\,  {\mathrm{d}}x \Big)^{\frac{1}{2}},
$$
where
$$
\int_{\mathbb{R}^N}|(-\Delta)^{s/2}u|^2{\mathrm{d}}x  = \int_{\mathbb{R}^{2N}}\frac{|u(x) - u(y)|^2}{|x - y|^{N + 2s}}\,{\mathrm{d}}x \,  {\mathrm{d}}y.
$$
Like the classical case, we define the space $\dot{H}^s(\R^N)$ as the completion of $C^{\infty}_c(\R^N)$ under the norm
$$
\|u\|^2 = \int_{\mathbb{R}^N}|(-\Delta)^{s/2}u|^2{\mathrm{d}}x  = \int_{\mathbb{R}^{2N}}\frac{|u(x) - u(y)|^2}{|x - y|^{N + 2s}}\,{\mathrm{d}}x \,  {\mathrm{d}}y.
$$

Generally, the fractional Laplacian  (see \cite{4} for example) is defined as
\begin{align*}
(-\Delta)^s u(x) \, &= \, C  (N,s)P.V.\int_{\mathbb{R}^N}\frac{u(x) - u(y)}{|x - y|^{N + 2s}}\,{\mathrm{d}}y
\\[2mm]
\, & =\,   C(N,s)\lim_{\epsilon\to 0}\int_{\mathbb{R}^N          \backslash        B_{\epsilon}(x)      }\frac{u(x) - u(y)}{|x - y|^{N + 2s}}\,{\mathrm{d}}y.
\end{align*}
For the sake of simplicity, we define for every $u\in \dot{H}^s(\R^N)$ the fractional $(-\Delta)^s$ as
$$
(-\Delta)^s u(x) = \int_{\mathbb{R}^N}\frac{u(x) - u(y)}{|x - y|^{N + 2s}}\,{\mathrm{d}}y.
$$
Our solutions will be found in the following weighted fractional Sobolev  space:
$$
\mathcal{D}^s_{V,\epsilon}(\mathbb{R}^N) = \left\{u\in \dot{H}^s(\R^N):\,\,u\in L^2(\mathbb{R}^N,V(x)\,{\mathrm{d}}x)\right\},
$$
endowed with the norm
$$
\|u\|_{\mathcal{D}^s_{V,\epsilon}(\mathbb{R}^N)} = \Big(\int_{\mathbb{R}^N}\epsilon^{2s}|(-\Delta)^{s/2}u|^2 + Vu^2\,{\mathrm{d}}x \Big)^{\frac{1}{2}}.
$$
For the potential term $V$, we assume that $V\in C\big(\mathbb{R}^N,[0,\infty)\big)$ and

$(\mathcal{A})$  there exist open bounded sets $\Lambda\subset\subset U$ with smooth boundaries $\partial\Lambda,\ \partial U$, such that
\begin{equation}\label{{Aeq1.3}}
0 < \lambda = \inf_{\Lambda}V < \inf_{U\backslash\Lambda}V.
\end{equation}
Without loss of generality, we assume that $0\in\Lambda$.

Now, with the notations and assumptions above at hand, we are in a position to state our main result:

\begin{theorem}\label{th1.1}
Let $N > 2s$, $s\in(0,1)$, $p\in (2 + \frac{2s}{N - 2s}, 2^*_s)$ and $V\in C\big(\mathbb{R}^N,[0,\infty)\big)$ satisfy the assumption $(\mathcal{A})$.
 Then there exists an $\epsilon_0>0$ such that problem \eqref{eq1.1} has a  positive solution $u_{\epsilon}\in \mathcal{D}^s_{V,\epsilon}(\mathbb{R}^N)$ if $\epsilon\in(0,\epsilon_0)$. Moreover, there exist  an $\alpha\in (2s/(p - 2),N - 2s)$ and a family of points $\{x_{\epsilon}\in \overline{\Lambda}:\epsilon\in(0,\epsilon_0)\}$ such that $u_{\epsilon}(x_{\epsilon}) = \sup_{x\in \Lambda}u_{\epsilon}(x)$,
 $$
 \lim_{\epsilon\to 0}V(x_{\epsilon}) = \min_{\Lambda}V(x)
 $$
 and
\begin{align*}
u_{\epsilon}(x)\le \frac{C\epsilon^{\alpha}}{\epsilon^{\alpha} + |x - x_{\epsilon}|^{\alpha}},
\end{align*}
where $C$ is positive constant.
\end{theorem}
Particularly, we solve the fractional version  Ambrosetti-Malchiodi question(\cite{AA}) completely, which in the local case $s=1$ have been settled recently (see \cite{BaNa,Byeon-Wang,Cao-Peng,MV}).

Our approach in this work follows the variational penalization scheme in \cite{APX}, which was first introduced in \cite{10}. Naturally, equation \eqref{eq1.1} is the Euler-Lagrange equation of the functional
$$
I_{\epsilon}(u):=\frac{1}{2}\int_{\R^N}(\epsilon^{2s}|(-\Delta)^{s/2}u|^2 + V|u|^2) - \frac{1}{p+1}\int_{\R^N}|u|^p.
$$
However, an elementary (but tedious) calculation shows that when $2<p<2^*_s$ and $V$ vanishes faster than $|x|^{-2s}$, the second integral is not finite in $\mathcal{D}^s_{V,\epsilon}(\R^N)$. del Pino and Felmer provided an effective way to overcome such difficulty: to modify the nonlinear term. According to the idea,  assuming that $V$ satisfies $\liminf_{|x|\to\infty}V(x)|x|^{2s}>0$, in \cite{APX} the penalized problem
$$
\epsilon^{2s}(-\Delta)^su_{\epsilon} + V(x)u_{\epsilon}=\chi_{\Lambda}u^{p-1}_{\epsilon} + \chi_{\R^N\backslash\Lambda}\min\{P_{\epsilon}(x)u_{\epsilon},u^{p-1}_{\epsilon}\}
$$
was considered, where $P_{\epsilon}(x)$ satisfies $P_{\epsilon}(x) = 0$ on $\Lambda$, $P_{\epsilon}(x)>0$ outside $\Lambda$ and $P_{\epsilon}(x)\le \epsilon^{\nu}V(x)$ outside $\Lambda$ with $\nu>0$ is a small parameter. A penalized solution then follows by the fact that the Euler-Lagrange functional corresponding to the penalized problem owns Mountain Pass geometry and is a class of $C^1$. One has then to show that $u^{p-2}_{\epsilon}\le P_{\epsilon}$ outside $\Lambda$. However, when $V$ is compactly supported this approach fails, because one should then have that the solution $u_{\epsilon}$ are compactly supported. In brief, the potential $V$ can not dominated the nonlinear term if it is compactly supported. Our key observation in this paper is that, in order to overcome this difficulty, we can modify the problem by the following fractional Hardy inequality: There exists a positive constant $C_{N,s}$ such that
\begin{equation}\label{eq1.4}
\int_{\R^N}\frac{|u(x)|^2}{|x|^{2s}}\,{\mathrm{d}}x\le C_{N,s}|(-\Delta)^{s/2}|^2_2
\end{equation}
for all $u\in \dot{H}^s(\R^N)$(see \cite{FS})

One difference between equations in local case $s=1$ and nonlocal case $0<s<1$ is the construction of penalized function. In fact, for a smooth function $f\in C^{\infty}_c(\R^N)$, one can not compute $(-\Delta)^sf$ as precisely as $-\Delta f$. Another difference is the energy estimates. The nonlocal effect makes us have to know the global $L^2$-norm information of penalized solution, which in \cite{APX} was given by the assumption $\liminf_{|x|\to 0}V(x)|x|^{2s}>0$. But, in the present paper,  the global $L^2$-norm information of penalized solution will be obtained by the Hardy inequality above (see \eqref{eq3.3} in the proof of Lemma \ref{le3.1} for more details).

\vspace{0.5cm}

  The paper is organized as follows: in Section \ref{s2}, we establish the penalized scheme and obtain a penalized solution $u_{\epsilon}$. In Section \ref{s3} we study the concentration phenomenon of $u_{\epsilon}$. In Section \ref{s4}, we prove the penalized solution $u_{\epsilon}$ solves the origin problem by constructing a special penalized function. In Section \ref{s5}, we give a short proof for the existence and concentration for equation \eqref{eq1.1} with more general nonlinear term.

\vspace{0.5cm}
\section{The penalized problem}\label{s2}

The following inequality exposes the relationship between $H^s(\R^N)$ and the Banach space $L^q(\R^N)$.
\begin{proposition}\label{wpr2.1} ({Fractional version of the Gagliardo$-$Nirenberg inequality \cite{4}})
For every $u\in H^s(\mathbb{R}^N)$,
\begin{align} \label{inequality2.1}
\|u\|_{q} \leq C \|(-\Delta)^{s/2}u\|^{\theta}_{2}\|u\|^{1 -\theta }_{2},
\end{align}
where $q\in [2,2^*_s]$ and $\theta $ satisfies $\frac{\theta}{2^*_s} + \frac{(1 -\theta)}{2} = \frac{1}{q}$.
\end{proposition}

By the proposition above, we know that $ H^s(\mathbb{R}^N)$ is continuously  embedded into $L^q(\mathbb{R}^N)$ for $\ q\in [2,2^*_s]$. Moreover, on bounded set, the embedding is compact (see \cite{4}), i.e.,
$$
H^s(\mathbb{R}^N)\subset\subset L^q_{loc}(\mathbb{R}^N)\ \text{compactly, if}\ q\in [1,2^*_s).
$$

Now we are going to modify the origin problem \eqref{eq1.1}.
According to the fractional Hardy inequality \eqref{eq1.4}, we choose a family of  penalized potentials $\mathcal{P}_{\epsilon}\in L^{\infty}(\mathbb{R}^N,[0,\infty))$ for $\epsilon > 0$ small in such a way that
\begin{eqnarray}\label{eq2.1}
\begin{split}
\mathcal{P}_{\epsilon}(x) = 0\ \text{for all}\ x\in\Lambda
\end{split}
\end{eqnarray}
{and}
\begin{align}
\ \lim\limits_{\epsilon \to 0}\sup_{\mathbb{R}^N\backslash \Lambda} \mathcal{P}_{\epsilon}(x)\epsilon^{-(2s + 3\kappa/2)}|x|^{2s + \kappa} = 0,
\end{align}
where $\kappa>0$ is a small parameter. Noting that by \eqref{eq1.4}, when $\epsilon>0$ is small enough, it holds
\begin{align*}
\int_{A}\mathcal{P}_{\epsilon}(x)|u|^2\le C_{N,s}\frac{\epsilon^{2s +\frac{3\kappa}{2}}}{\inf_{x\in(\R^N\backslash\Lambda)\cap A}{|x|^{\kappa}}}\int_{\R^N}|(-\Delta)^{s/2}u|^2
\end{align*}
where $C_{N,s}$ is the constant in \eqref{eq1.4}. This type of estimate plays a key role in the paper. Now we give the penalized problem according to the choice of $\mathcal{P}_{\epsilon}$:
\begin{equation}\label{eq2.2}
\epsilon^{2s}(-\Delta)^su + Vu = \chi_{\Lambda}(x)u^{p - 1}_{+} + \chi_{\mathbb{R}^N\backslash \Lambda}(x)\min(u^{p - 1}_{+}, \mathcal{P}_{\epsilon}(x)u_+).
\end{equation}
It is easy to check that if a solution $u_{\epsilon}$ of \eqref{eq2.2} satisfies
$$
u_{\epsilon}\le \mathcal{P}_{\epsilon}\ \ \text{on}\ \R^N\backslash\Lambda,
$$
then $u_{\epsilon}$ is a solution of \eqref{eq1.1}.

For simplicity, we define
$$
g_{\epsilon}(x,s): = \chi_{\Lambda}(x)s^{p - 1}_{+} + \chi_{\mathbb{R}^N\backslash \Lambda}(x)\min(s^{p - 1}_{+}, \mathcal{P}_{\epsilon}(x)s_+),
$$
$$G_{\epsilon}(x,t) = \int_{0}^{t}g_{\epsilon}(x,s) {\mathrm{d}}s,$$
$$
\mathbf{g}_{\epsilon}(u)(x) = g_{\epsilon}(x,u(x))
$$
and
$$
 \mathbf{G}_{\epsilon}(u)(x) = G_{\epsilon}(x,u(x)).
$$
Then the Euler-Lagrange functional of \eqref{eq2.2} 
is
$$
\mathbf{J}_{\epsilon}(u) = \frac{1}{2}\int_{\mathbb{R}^N}\big(\epsilon^{2s}|(-\Delta)^{s/2} u|^2 + V(x)|u|^2\big) - \int_{\mathbb{R}^N}\mathbf{G}_{\epsilon}(u).
$$

By  the Hardy inequality \eqref{eq1.4},  we know that $\mathbf{J}_{\epsilon}$ is well-defined on $\mathcal{D}^s_{V,\epsilon}(\R^N)$. In order to obtain a critical point to $\mathbf{J}_{\epsilon}$ via variational method, we need to verify that $\mathbf{J}_{\epsilon}$ is $C^1$ and satisfies the (P.S.) condition.
\begin{lemma}\label{le2.2}
$\mathbf{J}_{\epsilon}:\mathcal{D}^s_{V,\epsilon}(\R^N)\to \R$ is $C^1$ and satisfies (P.S.) condition.
\end{lemma}

\begin{proof}
We only need to show that the nonlinear term
$$
\mathcal{I}_{\epsilon}(u) = \int_{\R^N}\mathbf{G}_{\epsilon}(u)
$$
is $C^1$ since the other terms are obvious. For every $\varphi\in C^{\infty}_c(\R^N)$ and $0<|t|<1$, we have
$$
|\mathbf{G}_{\epsilon}(u + t\varphi) - \mathbf{G}_{\epsilon}(u)|/t\leq C\Big((|u|^p + |\varphi|^p)\chi_{\Lambda} + \mathcal{P}_{\epsilon}(|u|^2 + |\varphi|^2)\Big)\in L^1(\mathbb{R}^N).
$$
The existence of first order Gateaux derivative then follows  by Dominated Convergence Theorems.

Let $(u_n),\ u\in \mathcal{D}^s_{V,\epsilon}(\mathbb{R}^N)$ satisfy  $u_n\to u\in \mathcal{D}^s_{V,\epsilon}(\mathbb{R}^N)$. By Proposition \ref{wpr2.1}, the construction of $\mathcal{P}_{\epsilon}$, \eqref{eq1.4} and Dominated Convergence Theorem, we deduce that for every $\varphi\in H^{s}_{V,\epsilon}(\mathbb{\mathbb{R}^N})$ with $\|\varphi\|_{H^{s}_{V,\epsilon}(\mathbb{R}^N)} \leq 1$,
\begin{align*}
 &\quad\quad|\langle\mathcal{I}'_{\epsilon}(u_n) - \mathcal{I}'_{\epsilon}(u),\varphi\rangle|&\\
&\leq C\|u_n - u\|^p_{H^{s}_{V,\epsilon}(\mathbb{R}^N)} + \int_{\mathbb{R}^N\backslash{\Lambda}}\big|\min\{\mathcal{P}_{\epsilon},(u_n)^{p - 2}_+\}u_n\varphi - \min\{\mathcal{P}_{\epsilon},u^{p - 2}_+\}u\varphi \big| \,{\mathrm{d}}x
\\[2mm]
 & \leq  o_n(1) + \int_{\mathbb{R}^N\backslash{\Lambda}}\big|\min\{\mathcal{P}_{\epsilon},(u_n)^{p - 2}_+\}(u_n - u)\varphi\big|
 \\[2mm]
 &\ \ \ \ \ \ \ \ \ \ \ \ \ \ \qquad + \big|\min\{\mathcal{P}_{\epsilon},(u_n)^{p - 2}_+\} - \min\{\mathcal{P}_{\epsilon},u^{p - 2}_+\}\big|u\varphi \,{\mathrm{d}}x
 \\[2mm]
 & \leq o_n(1) + \int_{\mathbb{R}^N\backslash{\Lambda}}\big|\min\{\mathcal{P}_{\epsilon},(u_n)^{p - 2}_+\} - \min\{\mathcal{P}_{\epsilon},u^{p - 2}_+\}\big||u|^2\,{\mathrm{d}}x
 \\[2mm]
 & = o_n(1),
\end{align*}
which says that $\mathbf{J}_{\epsilon}$ is a class of $C^1$.

 Next we show that $\mathbf{J}_{\epsilon}$ satisfies (P.S.) condition, i.e., to prove that any sequence $\{u_n\}\subset  \mathcal{D}^s_{V,\epsilon}(\R^N)$ satisfying
$$
\mathbf{J}_{\epsilon}(u_n)\to c,\ \ \text{and}\ \ J'_{\epsilon}(u_n)\to 0
$$
is relatively compact.

It is standard to verify using the fact $p>2$ and the construction of $\mathcal{P}_{\epsilon}$ that $\{u_n\}$ is bounded in $\mathcal{D}^s_{V,\epsilon}(\R^N)$. By Proposition \ref{wpr2.1}, one has thus $u_n\to u$ in $L^{p}_{loc}(\R^N)$.

On the other hand, one has, for every $\sigma>0$, by the fractional Hardy inequality \eqref{eq1.4}, for $R>0$ large enough,
\begin{align*}
\int_{\R^N\backslash B_R(0)}\mathcal{P}_{\epsilon}(x)|u_n|^2 \,{\mathrm{d}}x  &= \sup_{x\in\R^N\backslash B_R(0)}(C^{-1}_{N,s}\mathcal{P}_{\epsilon}(x)|x|^{2s})\, C_{N,s}\, \int_{\R^N\backslash B_R(0)}\frac{|u_n|^2}{|x|^{2s}}\,{\mathrm{d}}x
\\[1mm]
&\le \sup_{x\in\R^N\backslash B_R(0)}(\mathcal{P}_{\epsilon}(x)|x|^{2s})|(-\Delta)^{s/2} u_n|^2_2\\[1mm]
&\le \frac{C}{R^{\kappa}}
< \sigma.
\end{align*}
Hence
\begin{align*}
  \limsup_{n\to\infty}\|u_n - u\|^2_{\mathcal{D}^s_{V,\epsilon}(\mathbb{R}^N)}
 &= \limsup\Big(\langle J'_{\epsilon}(u_n) - J'_{\epsilon}(u),u_n - u \rangle\\
 &\quad + \int_{\mathbb{R}^N}(\mathbf{g}_{\epsilon}(u_n) - \mathbf{g}_{\epsilon}(u))(u
_n(x) - u(x))\,{\mathrm{d}}x\Big)\\
  &\le \limsup_{n\to\infty}\Big(\int_{\mathbb{R}^N\backslash B_R(0)}\mathcal{P}_{\epsilon}(|u_n| + |u|)^2\,{\mathrm{d}}x\Big)^\frac{1}{2}\\
  &\le\sigma.
\end{align*}
Since $\sigma>0$ is arbitrary, this proves the lemma.

\end{proof}

It is standard to show that $\mathbf{J}_{\epsilon}$ owns Mountain Pass geometry. Hence, from Lemma \ref{le2.2}, we can use the min-max procedure in \cite{MW} to find a critical point for the functional $\mathbf{J}_{\epsilon}$.

\begin{lemma}\label{le2.3}
The Mountain Pass value
$$
c_{\epsilon}: = \inf_{\gamma\in\Gamma_{\epsilon}}\max_{t\in[0,1]}\mathbf{J}_{\epsilon}(\gamma(t))
$$
can be achieved by a positive function $u_{\epsilon}$ which is $C^{1,\gamma}$ for some $\gamma\in(0,1)$ and is a solution to the penalized problem \eqref{eq2.2},
where
$$
\Gamma_{\epsilon}:=\big\{\gamma\in C\big(\, [0,1],\, \mathcal{D}^s_{V,\epsilon}(\R^N)\, \big)|\gamma(0) = 0, \mathbf{J}_{\epsilon}(\gamma(1))<0\big\}.
$$
\end{lemma}

\begin{proof}
The existence of $u_{\epsilon}$ can be checked by the Mountain Pass Lemma in \cite{MW}. By the regularity argument in \cite[Appendix D]{20}, $u_{\epsilon}$ is $C^{1,\gamma}$ for some $\gamma\in(0,1)$. {Testing the penalized equation \eqref{eq2.2} with $(u_{\epsilon})_{-}$ and integrating, we can see that $u_{\epsilon}\geq 0$.}

{Supposing to the contrary that there exists  $x_0\in \mathbb{R}^N$ such that $u_{\epsilon}(x_0) = 0$, we then have
$$
0 = (-\Delta)^s{u_{\epsilon}}(x_0) + V(x_0)u_{\epsilon}(x_0) < 0,
$$
which is a contradiction. Therefore, $u_{\epsilon}>0$.}

\end{proof}

\section{Concentration}\label{s3}

In this section, we prove that the solution  $u_{\epsilon}$ obtained by Lemma \ref{le2.3} will concentrate at a local minimum of $\Lambda$ as $\epsilon\to 0$ {via comparing the energy $c_{\epsilon}$ in Lemma \ref{le2.3} with the least energy of the limiting problem corresponding to \eqref{eq1.1}.} This concentration can help to linearize the penalized equation \eqref{eq2.2}\,(see \eqref{AAeq4.2} below).

\subsection{The limiting problem}

\noindent  {For $a > 0$, the limiting problem associated to  \eqref{eq1.1} is
$$
(-\Delta)^s v + a v = |v|^{p - 2}v,
\eqno(\mathcal{P}_{a})
$$
whose Lagrange-Euler  functional $J_{a}: H^s(\mathbb{R}^N)\to \mathbb{R}$ is defined  as
$$
J_{a}(v) = \frac{1}{2}\int_{\mathbb{R}^N}(|(-\Delta)^{s/2} v|^2 + a |v|^2)dx - \frac{1}{p}\int_{\mathbb{R}^N}|v|^p dx.
$$}

{From \cite{20,FQT}, we know that the limiting problem has a positive ground state. We define the limiting energy by
\begin{equation}\label{eq3.1}
\mathcal{C}(a) = \inf_{v\in H^s(\mathbb{R}^N)\backslash\{0\}}\max_{t > 0}J_{a}(tv).
\end{equation}
Since for every $v\in H^s(\mathbb{R}^N),\ J_{a}(|v|) \leq J_{a}(v)$, by the  density of $C_0^\infty(\mathbb{R}^N)$ in   $H^s(\mathbb{R}^N)$, we have
\begin{equation}\label{eq3.2}
  \mathcal{C}(a) = \inf_{{C^{\infty}_c(\mathbb{R}^N)\backslash\{0\}}\atop{v\geq 0}}\max_{t > 0}J_{a}(tv).
\end{equation}
The following proposition is obvious.
\begin{proposition}
Let $a > 0$ and $v\in H^s(\mathbb{R}^N)$. Define
$$
v_{a}(y) = a^{\frac{1}{p - 2}}v(a^{1/2s}y).
$$
Then
$$
J_{a}(v_{a}) = a^{\frac{p}{p - 2} - \frac{N}{2s}}J_1(v_1).
$$
In particular, $v$ is a solution of $(\mathcal{P}_1)$  if and only if $v_a$ is a solution of $(\mathcal{P}_{a})$ and
$$
\mathcal{C}({a}) = \mathcal{C}({1})a^{\frac{p}{p - 2} - \frac{N}{2s}}.
$$
Moreover, $\mathcal{C}({a})$ is continuous and increasing in $(0,\infty)$.
\end{proposition}}

{Now we are going to compare the energy $c_{\epsilon}$ with $\mathcal{C}(V(x)),x\in\Lambda$. Firstly,
we give the upper bound of the energy $c_{\epsilon}$:
\begin{lemma}\label{GGle3.2}(Upper bound of the energy) we have
\begin{equation}\label{eq3.4}
 \limsup_{\epsilon\to 0}\frac{c_{\epsilon}}{\epsilon^N}\leq \inf_{x\in\Lambda}\mathcal{C}(V(x)).
\end{equation}
\end{lemma}}

\begin{proof}
For any given $x_0 \in \Lambda$ and nonnegative function $v \in C^{\infty}_c (\mathbb{R}^N)\backslash\{0\}$, we define
$$
v_{\epsilon}(x) = v\Big(\frac{x - x_0}{\epsilon}\Big).
$$
Obviously, $supp \,\,v \subset \Lambda$ when $\epsilon$ is small enough. It is easy  to check that $\gamma_{\epsilon}(t) = tT_0v_{\epsilon}\in \Gamma_{\epsilon}$ for some $T_0$ large enough. So
\begin{align*}
  \frac{c_{\epsilon}}{\epsilon^N}&\leq \max_{t\in [0,1]}J_{\epsilon}(\gamma_{\epsilon}(t))\leq \max_{t>0}J_{V(x_0)}(tv) + o_\epsilon(1),
\end{align*}
and
$$
\limsup_{\epsilon\to 0}\frac{c_{\epsilon}}{\epsilon^N}\leq \inf_{{C^{\infty}_c(\mathbb{R}^N)\backslash\{0\}}\atop{v\geq 0}}\max_{t > 0}J_{V(x_0)}(tv) \leq \mathcal{C}(V(x_0)),
$$
which implies our conclusion.
\end{proof}

Secondly, we  give the lower estimate on energy $c_{\epsilon}$.

\begin{lemma}\label{le3.1}
 Let $\{ \epsilon_n \}_{n\in \mathbb{N}}$ be a sequence of positive numbers converging to $0$,  $\{ u_n\} $ be a sequence of critical points given by Lemma \ref{le2.3},
  and for $j\in {1,2,\cdots,k}$,  $\{ x^j_n\} $ be a sequence in $\mathbb{R}^N$  converging  to $x^j_{*}\in \mathbb{R}^N$. If
$$
\limsup_{n\to\infty}\frac{1}{\epsilon^N_n}\int_{\mathbb{R}^N} \big\{ \epsilon^{2s}_n|(-\Delta)^{s/2}u_n|^2 + V|u_n|^2 \big\} \,{\mathrm{d}}x  < \infty,
$$

$$
V(x^j_{*}) > 0\ \text{and}\ \lim\limits_{n\to \infty}\frac{|x^i_n - x^j_n|}{\epsilon_n} = \infty\ \ \text{if}\ i\neq j \quad \text{for all}\ i,j = 1,2,\cdots,k,
$$
and for some $\rho > 0$,
$$
\liminf_{n\to \infty}\|u_n\|_{L^{\infty}(B_{\epsilon_n\rho}(x^j_n))} > 0 \quad     \text{for all}\,  j = 1,2,\cdots k,
$$
then $x^j_{*}\in \bar{\Lambda}$ and
$$
\liminf_{n\to \infty}\frac{J_{\epsilon_n}(u_n)}{{\epsilon^N_n}}\geq\sum_{j = 1}^{k}\mathcal{C}(V(x^j_*)).
$$
\end{lemma}


Some details in the proof of this lemma will be omitted since they are similar to Proposition 3.4 in \cite{APX}, and more attentions will be paid on the difference caused by the fast decaying potential $V$ (see \eqref{eq3.3} below for more details).

\begin{proof}
 For $j\in \{1,\ldots,k\}$, the rescaling  function $v^j_n\in H^s_{loc}(\mathbb{R}^N)$ defined by
$$
v^j_n(y) = u_n(x^j_n + \epsilon_ny)
$$
satisfies weakly the rescaled equation
$$
(-\Delta)^{s} v^j_n + V^j_n v^j_n = \textsf{g}^j_n(v^j_n)\ \text{in}\ \mathbb{R}^N,
$$
where  $V^j_n(y) = V(x^j_n + \epsilon_n y),\ \textsf{g}^j_n(v^j_n) = g_{\epsilon_n}(x^j_n + \epsilon_ny, v^j_n).
$ By the assumption, $$\sup_{n\in \N}\|u_n\|_{\mathcal{D}^s_{V,\epsilon}(\R^N)}<+\infty, $$
 the sequence $\{ v^j_n\}_{n\ge 1}$ is bounded in $H^s_{loc}(\R^N)$. By the regularity assertion in Appendix D of \cite{20}, the Liouville-type Lemma in Lemma 3.3 of \cite{APX}, we can conclude that there exist $x^j_*\in\bar{\Lambda}$ and a nonnegative $v^j_*\in H^s(\R^N)\backslash\{0\}$ such that $v^j_n\to v^j_*$ in $H^s_{loc}(\R^N)\cap C^1_{loc}(\R^N)$, $x^j_n\to x^j_*$ and
$$
(-\Delta)^s v^j_* + V(x^j_*) v^j_* = (v^j_*)^{p - 1}.
$$
Since $v^j_* \geq 0$, we see
\begin{align*}
  &\liminf_{n\to\infty}\frac{1}{{\epsilon^N_n}}\int_{B_{\epsilon_nR}(x^j_n)}  \Big\{ \, \frac{1}{2}(\epsilon^{2s}_n|(-\Delta)^{s/2}u_n|^2 + V|u_n|^2) - \mathbf{G}_{\epsilon_n}(u_n)   \, \Big\}\,{\mathrm{d}}x \\
 \geq &\mathcal{C}(V(x^j_*)) - C\int_{\mathbb{R}^N\backslash B_R} \Big\{ \,|  (-\Delta)^{s/2}v^j_*|^2 + V(x^j_*)|v^j_*|^2 \, \Big\}\,{\mathrm{d}}x,
\end{align*}
where $C > 0$ is a universal constant.

In order to study the integral outside $B_{\epsilon_nR}(x^j_n)$, we choose cut-off function  $\eta(x) \in C^{\infty}(\mathbb{R}^N)$ such that $0\leq \eta\leq 1,\ \eta = 0$ if  $ 0 \leq |x| \leq 1$ and $\eta = 1$ when  $     |x| \geq  2$. Define
$$
\psi_{n,R}(x) = \prod^{k}_{j = 1}\eta\Big(\frac{x - x^j_n}{\epsilon_n R}\Big).
$$
Since $u_n$ is a solution to the penalized problem $\eqref{eq2.2}$, we have by taking $\psi_{n,R}u_n$ as a test function in the penalized problem $\eqref{eq2.2}$ that
\begin{align}\label{teq}
 \nonumber &\int_{\mathbb{R}^N\backslash\cup^{k}_{j = 1}B_{\epsilon_nR}(x^j_n)}  \Big\{ \epsilon^{2s}_n   \psi_{n,R}|(-\Delta)^{s/2}u_n|^2 + V\psi_{n,R}|u_n|^2  \Big\} \,{\mathrm{d}}x
 \\[1mm]
   = &\int_{\mathbb{R}^N\backslash\cup^{k}_{j = 1}B_{\epsilon_nR}(x^j_n)}\mathbf{g}_{\epsilon_n}(u_n)u_n\psi_{n,R} \,{\mathrm{d}}x
   \\[1mm]
\nonumber  &-\int_{\mathbb{R}^N\backslash\cup^{k}_{j = 1}B_{\epsilon_nR}(x^j_n)}\int_{\mathbb{R}^N}\frac{u_n(y)(\psi_{n,R}(x) - \psi_{n,R}(y))(u_n(x) - u_n(y))}{|x - y|^{N + 2s}}\,{\mathrm{d}}y\,{\mathrm{d}}x
\\[1mm]
\nonumber  :=& \int_{\mathbb{R}^N\backslash\cup^{k}_{j = 1}B_{\epsilon_nR}(x^j_n)}\mathbf{g}_{\epsilon_n}(u_n)u_n\psi_{n,R}     \,{\mathrm{d}}x  + R_n.
\end{align}
Hence
\begin{align*}
  &\ds\int_{\mathbb{R}^N\backslash\cup^{k}_{j = 1}B_{\epsilon_nR}(x^j_n)}   \Big\{\,  \frac{1}{2}(\epsilon^{2s}_n|(-\Delta)^{s/2}u_n|^2 + V|u_n|^2) - \mathbf{G}_{\epsilon_n}  \,  \Big\}  \,{\mathrm{d}}x
  \\
  \geq & \frac{1}{2}\int_{\mathbb{R}^N\backslash\cup^{k}_{j = 1}B_{\epsilon_nR}(x^j_n)} \Big\{\,  \epsilon^{2s}_n\psi_{n,R}|(-\Delta)^{s/2}u_n|^2 + V\psi_{n,R}|u_n|^2 - \mathbf{g}_{\epsilon_n}(u_n)u_n \,  \Big\}  \,{\mathrm{d}}x
  \\
   =& -\frac{\epsilon^{2s}_n}{2}R_n + \int_{\mathbb{R}^N\backslash\cup^{k}_{j = 1}B_{\epsilon_nR}(x^j_n)}\mathbf{g}_{\epsilon_n}(u_n)(\psi_{n,R} - 1)u_n  \,{\mathrm{d}}x.
\end{align*}
By scaling, $B_{\epsilon_nR}(x^j_n)\cap B_{\epsilon_nR}(x^l_n) = \emptyset$ if $n$ is large enough. Hence, by the fact that $v^j_n\to v^j_*$ in $L^{p}_{loc}(\mathbb{R}^N)$, we have
$$
\lim\sup_{n\to\infty}\Big|\frac{1}{2}\sum_{j = 1}^{k}\int_{B_{2R}\backslash B_{R}}\textsf{g}^j_n(v^j_n)v^j_n   \,{\mathrm{d}}x \Big|      = \Big|\frac{1}{2}\sum_{j = 1}^{k}\int_{B_{2R}\backslash B_{R}}(v^j_*)^p    \,{\mathrm{d}}x \Big| = o_R(1).
$$

Now we estimate $R_n$. A change of variable tells us
\begin{align}\label{eqi1}
R_n & = \int_{\mathbb{R}^N}u_n(y)\,{\mathrm{d}}y\int_{\mathbb{R}^N}\frac{(u_n(x) - u_n(y))(\psi_{n,R}(x) - \psi_{n,R}(y))}{|x - y|^{N + 2s}}\,{\mathrm{d}}x\\
\nonumber    & = \epsilon^{N - 2s}_n\sum_{l = 1}^k\int_{\mathbb{R}^N}v^l_n(y)\beta^l_n(y)\,{\mathrm{d}}y\int_{\mathbb{R}^N}\frac{\alpha^l_n(x)(v^l_n(x) - v^l_n(y))(\eta_R(x) - \eta_R(y))}{|x - y|^{N + 2s}}\,{\mathrm{d}}x,
\end{align}
where the functions $\beta^l_n$ and $\alpha^l_n$ are defined skillfully as
$$
\beta^l_n(y) = \prod^{l - 1}_{s = 0}\eta\Big(\frac{y}{R} + \frac{x^l_n - x^s_n}{\epsilon_n R}\Big),\,\,\,\beta^0_n(y)\equiv 1
$$
and
$$
\alpha^l_n(x) = \prod^{k}_{s = l + 1}\eta\Big(\frac{x}{R} + \frac{x^l_n - x^s_n}{\epsilon_n R}\Big),\,\,\, \alpha^k_n(x) \equiv 1.
$$
Following, we have
\begin{align*}
\epsilon^{2s - N}_nR_n & = \sum_{l = 1}^k\int_{B_{2R}} v^l_n(y) \beta^l_n(y) \,{\mathrm{d}}y \int_{B^c_{2R}}\frac{\alpha^l_n(x)(v^l_n(x) - v^l_n(y))(\eta_R(x) - \eta_R(y))}{|x - y|^{N + 2s}}\,{\mathrm{d}}x
\\[1mm]
   &\quad + \sum_{l = 1}^k \int_{B^c_{2R}}v^l_n(y)\beta^l_n(y)\,{\mathrm{d}}y  \int_{B_{2R}}\frac{\alpha^l_n(x)(v^l_n(x) - v^l_n(y))(\eta_R(x) - \eta_R(y))}{|x - y|^{N + 2s}}\,{\mathrm{d}}x
   \\[1mm]
   &\quad + \sum_{l = 1}^k\int_{B_{2R}}v^l_n(y)\beta^l_n(y)\,{\mathrm{d}}y   \int_{B_{2R}}\frac{\alpha^l_n(x)(v^l_n(x) - v^l_n(y))(\eta_R(x) - \eta_R(y))}{|x - y|^{N + 2s}}\,{\mathrm{d}}x\\[1mm]
   &:=R^{(1)}_n + R^{(2)}_n + R^{(3)}_n.
\end{align*}
By the choice of $\eta$ and  $\lim\limits_{n\to\infty}\frac{|x^l_n - x^s_n|}{\epsilon_n} = \infty$ if $l\neq s$, for $n$ large, we have
\begin{align*}
R^{(1)}_n&= \sum_{l = 1}^k\int_{B_{2R}\backslash B_R}v^l_n(y)\beta^l_n(y)\,{\mathrm{d}}y\int_{B^c_{2R}}\frac{\alpha^l_n(x)(v^l_n(x) - v^l_n(y))(1 - \eta_R(y))}{|x - y|^{N + 2s}}\,{\mathrm{d}}x\\[1mm]
       &\quad + \sum_{l = 1}^k\int_{{B_R}}v^l_n(y)\beta^l_n(y)\,{\mathrm{d}}y\int_{B^c_{2R}}\frac{\alpha^l_n(x)(v^l_n(x) - v^l_n(y))}{|x - y|^{N + 2s}}\,{\mathrm{d}}x\\[1mm]
       &= \sum_{l = 1}^k\int_{B_{2R}\backslash B_R}v^l_n(y)\,{\mathrm{d}}y\int_{B^c_{2R}}\frac{\alpha^l_n(x)(v^l_n(x) - v^l_n(y))(1 - \eta_R(y))}{|x - y|^{N + 2s}}\,{\mathrm{d}}x\\[1mm]
       &\quad + \sum_{l = 1}^k\int_{{B_R}}v^l_n(y)\,{\mathrm{d}}y\int_{B^c_{2R}}\frac{\alpha^l_n(x)(v^l_n(x) - v^l_n(y))}{|x - y|^{N + 2s}}\,{\mathrm{d}}x\\[1mm]
       :&=R^{(11)}_n + R^{(12)}_n
\end{align*}
and
\begin{align*}
        R^{(2)}_n
          =& \sum_{l = 1}^k\int_{B^c_{2R}}v^l_n(y)\beta^l_n(y)\,{\mathrm{d}}y\int_{B_{2R}\backslash B_R}\frac{\alpha^l_n(x)(v^l_n(x) - v^l_n(y))(\eta_R(x) - 1)}{|x - y|^{N + 2s}}\,{\mathrm{d}}x\\[1mm]
         & - \sum_{l = 1}^k\int_{B^c_{2R}}v^l_n(y)\beta^l_n(y)\,{\mathrm{d}}y\int_{B_R}\frac{\alpha^l_n(x)(v^l_n(x) - v^l_n(y))}{|x - y|^{N + 2s}}\,{\mathrm{d}}x \\[1mm]
         =& \sum_{l = 1}^k\int_{B^c_{2R}}v^l_n(y)\beta^l_n(y)\,{\mathrm{d}}y\int_{B_{2R}\backslash B_R}\frac{(v^l_n(x) - v^l_n(y))(\eta_R(x) - 1)}{|x - y|^{N + 2s}}\,{\mathrm{d}}x\\[1mm]
         & - \sum_{l = 1}^k\int_{B^c_{2R}}v^l_n(y)\beta^l_n(y)\,{\mathrm{d}}y\int_{B_R}\frac{v^l_n(x) - v^l_n(y)}{|x - y|^{N + 2s}}\,{\mathrm{d}}x \\[1mm]
                  :=&R^{(21)}_n + R^{(22)}_n.
\end{align*}
Also, for large $n$,
\begin{align*}
R^{(3)}_n =&\sum_{l = 1}^k\int_{B_{2R}\backslash{B_R}}v^l_n(y)\,{\mathrm{d}}y\int_{B_{2R}}\frac{(v^l_n(x) - v^l_n(y))(\eta_R(x) - \eta_R(y))}{|x - y|^{N + 2s}}\,{\mathrm{d}}x\\[1mm]
& + \sum_{l = 1}^k\int_{B_{2R}}v^l_n(y)\,{\mathrm{d}}y\int_{B_{2R}\backslash{B_R}}\frac{(v^l_n(x) - v^l_n(y))(\eta_R(x) - \eta_R(y))}{|x - y|^{N + 2s}}\,{\mathrm{d}}x\\[1mm]
&+ \sum_{l = 1}^k\int_{B_{2R}\backslash{B_R}}v^l_n(y)\,{\mathrm{d}}y\int_{B_{2R}\backslash B_{R}}\frac{(v^l_n(x) - v^l_n(y))(\eta_R(x) - \eta_R(y))}{|x - y|^{N + 2s}}\,{\mathrm{d}}x\\[1mm]
:=& R^{(31)}_n + R^{(32)}_n + R^{(33)}_n.
\end{align*}

For $|R^{(i2)}_n|,\ i = 1,2$, it holds
\begin{align*}
\nonumber    &\limsup_{n\to\infty}  |R^{(i2)}_n|  \leq CR^{-2s} + \limsup_{n\to\infty}2\sum_{l = 1}^k\int_{B^c_{2R}}\,{\mathrm{d}}y\int_{B_{R}}\frac{ (v^l_n(y))^2}{(|y| - R)^{N + 2s}}\,{\mathrm{d}}x.
\end{align*}
By the fractional Hardy inequality \eqref{eq1.4} and letting $\widetilde{R} = R^{\frac{N + 1}{N}}$, we find
\begin{align}\label{eq3.3}
\nonumber&\quad \limsup_{n\to\infty}\int_{B^c_{2R}}\big(v^l_n(y))^2\,{\mathrm{d}}y\int_{B_R}\frac{1}{|x - y|^{N + 2s}}\,{\mathrm{d}}x\\[1mm]
\nonumber&\leq C\limsup_{n\to\infty}\int_{B^c_{2R}}\big(v^l_n(y))^2\frac{R^N}{|y|^{N + 2s}}\,{\mathrm{d}}y\\[1mm]
\nonumber& \le C\limsup_{n\to\infty}\int_{B_{\widetilde{R}}\backslash B_{2R}}\big(v^l_n(y))^2\frac{R^N}{|y|^{N + 2s}}\,{\mathrm{d}}y + C\limsup_{n\to\infty} \int_{B^c_{\widetilde{R}}}\big(v^l
_n(y))^2\frac{R^N}{|y|^{N + 2s}}\,{\mathrm{d}}y\\[1mm]
&\leq C\limsup_{n\to\infty}\int_{B_{\widetilde{R}}\backslash B_{2R}}\big(v^l_n(y))^2\,{\mathrm{d}}y + C \limsup_{n\to\infty}\int_{B^c_{\widetilde{R}}}\frac{(v^l_n(y))^2}{|y|^{2s}}\frac{R^N}{|y|^{N}}\,{\mathrm{d}}y\\[1mm]
\nonumber &\leq  C\int_{B_{\widetilde{R}}\backslash B_{2R}}\big(v^l_*(y))^2\,{\mathrm{d}}y + \frac{C}{R}\\[1mm]
\nonumber& = o_R(1).
\end{align}

For $R^{(11)}_n$, by the estimates of  $R^{(i2)}_n$ above and the similar estimates of Proposition 3.4 in \cite{APX}, we have

%
%
%
\begin{align}\label{eq3.8}
&\limsup_{n\to\infty}  |R^{(11)}_n| \nonumber
\\[1mm]
&\leq
  \limsup_{n\to\infty}  \sum_{l = 1}^k\int_{B_{2R}\backslash B_R}\,{\mathrm{d}}y\int_{B^c_{2R}}\frac{|v^l_n(x) - v^l_n(y)|^2}{|x - y|^{N + 2s}}\,{\mathrm{d}}x \nonumber
  \\[1mm]
 &
 + \limsup_{n\to\infty}\sum_{l = 1}^k \int_{B_{2R}\backslash B_R}(v^l_n(y))^2\,{\mathrm{d}}y \int_{B^c_{2R}}\frac{(1 - \eta_{R}(y))^2}{|x - y|^{N + 2s}}\,{\mathrm{d}}x \nonumber
  \\[1mm]
 & \leq  \limsup_{n\to\infty}  \sum_{l = 1}^k  \int_{B_{2R}\backslash B_R}\,{\mathrm{d}}y\int_{B^c_{2R}\cap B_{4R}}\frac{|v^l_n(x) - v^l_n(y)|^2}{|x - y|^{N + 2s}}\,{\mathrm{d}}x\nonumber
  \\[1mm]
 &+ \limsup_{n\to\infty}\sum_{l = 1}^k \int_{B_{2R}\backslash B_R}\,{\mathrm{d}}y\int_{B^c_{4R}}\frac{|v^l_n(x) - v^l_n(y)|^2}{|x - y|^{N + 2s}}\,{\mathrm{d}}x + C\limsup_{n\to\infty}\sum_{l = 1}^k\int_{B_{2R}\backslash B_R}(v^l_n(y))^2\,{\mathrm{d}}y\nonumber
 \\[1mm]
 & \leq  \limsup_{n\to\infty}\sum_{l = 1}^k\int_{B_{2R}\backslash B_R}\,{\mathrm{d}}y\int_{B^c_{2R}\cap B_{4R}}\frac{|v^l_n(x) - v^l_n(y)|^2}{|x - y|^{N + 2s}}\,{\mathrm{d}}x + o_R(1)\nonumber
 \\[1mm]
& \leq C\int_{B_{2R}\backslash B_R}\,{\mathrm{d}}y\int_{R^N}\frac{|v^l_*(x) - v^l_*(y)|^2}{|x - y|^{N + 2s}}\,{\mathrm{d}}x + o_R(1).
\end{align}
Similarly, we get
\begin{align*}
    \limsup_{n\to\infty}  |R^{(21)}_n| &\leq o_R(1)
\end{align*}
and
\begin{align*}
  \limsup_{n\to\infty}|R^{(31)}_n| + |R^{(32)}_n| + |R^{(33)}_n|\leq o_R(1).
\end{align*}
Therefore
\begin{equation}\label{impotrant}
\epsilon^{2s - N}|R_n| \leq o_R(1).
\end{equation}
Finally, letting $R\to\infty$, we  have
$$
\liminf_{n\to \infty}\frac{J_{\epsilon_n}(u_n)}{{\epsilon^N_n}}\geq\sum_{j = 1}^{k}\mathcal{C}(V(x^j_*)).
$$
\end{proof}


At last, as a result of Lemma \ref{le3.1}, we have the  following concentration for the solution $u_{\epsilon}$.

\begin{lemma}\label{le2.5}
 Let $\rho > 0$. There exists a family of points $\{x_{\epsilon}\}\subset\Lambda$ such that
\begin{eqnarray*}
  &&(i)\ \liminf_{\epsilon\to 0}\|u_{\epsilon}\|_{L^{\infty}(B_{\epsilon\rho}(x_{\epsilon}))}> 0,\\
  &&(ii)\ \lim\limits_{\epsilon\to 0}V(x_{\epsilon}) = \inf_{\Lambda}V(x),\\
  &&(iii)\ \lim\limits_{{R\to \infty}\atop{{\epsilon\to 0}}}\|u_{\epsilon}\|_{L^{\infty}(U\backslash B_{\epsilon R}(x_{\epsilon}))} =  0.
\end{eqnarray*}
\end{lemma}

\begin{proof}
It is easy to verify using the construction of $\mathcal{P}_{\epsilon}$ that
$$
\liminf_{\epsilon\to 0}\|u_{\epsilon}\|_{L^{\infty}(\Lambda)} > 0.
$$
Then by the regularity assertion in Appendix D of \cite{20}, we get the existence of $x_{\epsilon}\in\bar{\Lambda}$. We assume that $x_{\epsilon}\to x_*$.

By Lemma \ref{le3.1}, it holds
$$
\liminf_{\epsilon\to 0}\frac{\mathbf{J}_{\epsilon}(u_{\epsilon})}{\epsilon^N}\ge {\mathcal{C}}(V(x_*)),
$$
But,
$$
\inf_{\Lambda}\mathcal{C}(V(x))\ge \limsup_{\epsilon\to 0}\frac{\mathbf{J}_{\epsilon}(u_{\epsilon})}{\epsilon^N}.
$$
Hence by the monotonicity of $\mathcal{C}(\cdot)$, $V(x_*) = \inf_{\Lambda}V(x).$

Arguing by contradiction, if $(iii)$ is  not hold, we can infer from Lemma \ref{le3.1} that
$$
\inf_{\Lambda}\mathcal{C}(V(x))\ge\liminf_{\epsilon\to 0}\frac{\mathbf{J}_{\epsilon}(u_{\epsilon})}{\epsilon^N}\ge 2\inf_{\Lambda}\mathcal{C}(V(x)),
$$
which is a contradiction. The proof is then completed.

\end{proof}

\vspace{0.5cm}

\section{Back to the original problem}\label{s4}

\noindent In this section, we show that $u_{\epsilon}^{p - 2}\leq \mathcal{P}_{\epsilon}$ by comparison principle.

{By Lemma \ref{le2.5} and the $L^{\infty}$ estimate in \cite[Appendix D]{20}, we can assume that there exists a positive constant $C_{\infty}$ such that
\begin{equation}\label{eq4.0}
\sup_{x\in \Lambda}|u_{\epsilon}(x)|\leq C_{\infty}
\end{equation}
if $\epsilon$ is small enough,
from which we can linearize the penalized problem \eqref{eq2.2} outside small balls as follows.}

\begin{proposition}\label{pr4.1} For $\epsilon > 0$ small enough and $\delta\in (0,1)$, there exist $R > 0$ and $x_{\epsilon}\in \Lambda$  such that
\begin{equation}\label{AAeq4.2}
  \left\{
    \begin{array}{ll}
      \epsilon^{2s}(-\Delta)^{s} u_{\epsilon} + (1 - \delta)Vu_{\epsilon}\leq \mathcal{P}_{\epsilon}(x)u_{\epsilon}(x), & \text{in}\ \mathbb{R}^N\backslash  B_{R\epsilon}(x_{\epsilon}), \vspace{2mm}\\
      u_{\epsilon}\leq C_{\infty} & \text{in}\ \Lambda.
    \end{array}
  \right.
\end{equation}
\end{proposition}
\noindent Letting $v_{\epsilon}(x) = u_{\epsilon}(\epsilon x + x_{\epsilon})$, it is easy to check that
\begin{equation}\label{AAeq4.3}
  \left\{
    \begin{array}{ll}
      (-\Delta)^{s} v_{\epsilon} + (1 - \delta)V_{\epsilon}(x)v_{\epsilon}\leq \widetilde{P}_{\epsilon}(x)u_{\epsilon}(x), & \text{in}\ \mathbb{R}^N\backslash  B_{R}(0), \vspace{2mm}\\
      v_{\epsilon}\leq C_{\infty} & \text{in}\ B_{R}(0),
    \end{array}
  \right.
\end{equation}
where $V_{\epsilon}(\cdot) = u_{\epsilon}(\epsilon \cdot + x_{\epsilon})$ and $\widetilde{P}_{\epsilon}(\cdot) = \mathcal{P}_{\epsilon}(\epsilon \cdot + x_{\epsilon})$.

 Now we construct a suitable sup-solution to equation \eqref{AAeq4.3}. Let ${ \tilde{\eta}_\beta }(|x|)  :\R^N\to [0,1]$ be a smooth non-increasing function with ${ \tilde{\eta}_{\beta}  }(s) \equiv 1$         when  $ s \in [-1,1]$ and $\tilde{\eta}_{\beta}\equiv 0$ when  $  s\in (-\infty,-1 - \beta)\bigcup(1 + \beta,+\infty)$, where $\beta>0$ is a small parameter.
Define $\eta_{\beta}(|x|) = \tilde{\eta}_{\beta}(R|x|)$. It is worth mentioning  that the non-increasing of $\eta$ is important, see \eqref{Ieq1} below.
Setting $0<\alpha< N -2s$ and denoting  $$f^{\beta}_{\alpha}(x) = \eta_{\beta}(x)\frac{1}{R^{\alpha}} + (1 - \eta_{\beta}(x))\frac{1}{|x|^{\alpha}},$$  we have:
\begin{proposition}\label{pr3.2}
Let $\epsilon>0$ be small enough. Then for every $x\in \mathbb{R}^N\backslash B_R(0)$, it holds
\begin{equation}\label{AAeq4.5}
      (-\Delta)^{s} f^{\beta}_{\alpha} + V_{\epsilon}(x)f^{\beta}_{\alpha} - \widetilde{P}_{\epsilon}(x)f^{\beta}_{\alpha}\geq 0.
\end{equation}
\end{proposition}

\begin{proof}
For $x\in \R^N\backslash B_{2R}(0)$, letting $\beta>0$ be small enough, it holds
\begin{align*}
&(-\Delta)^sf^{\beta}_{\alpha}(x)\\
                  & = \int_{{ B_R(0)}}\frac{|x|^{-\alpha} - {R^{-\alpha}}}{|x - y|^{N + 2s}}\,{\mathrm{d}}y + \int_{\mathbb{R}^N\backslash{{ B_R(0)}}}\frac{|x|^{-\alpha} - |y|^{-\alpha}}{|x - y|^{N + 2s}}\,{\mathrm{d}}y + \int_{\mathbb{R}^N\backslash{{ B_R(0)}}}\frac{\eta(y)|y|^{-\alpha} - \eta(y){R^{-\alpha}}}{|x - y|^{N + 2s}}\,{\mathrm{d}}y
                  \\[1mm]
                  & = (-\Delta)^s(|\cdot|^{\alpha})(x) +  \int_{{{ B_R(0)}}}\frac{|y|^{-\alpha} - {R^{-\alpha}}}{|x - y|^{N + 2s}}\,{\mathrm{d}}y + \int_{\mathbb{R}^N\backslash{{ B_R(0)}}}\frac{\eta(y)|y|^{-\alpha} - \eta(y){R^{-\alpha}}}{|x - y|^{N + 2s}}\,{\mathrm{d}}y
                  \\[1mm]
                  & \ge \frac{f^{\beta}_{\alpha}(x)}{|x|^{2s}}\Big(\int_{(B_1(0))^c}\frac{1- |z|^{\alpha}}{|z|^{N - 2s}|z - \vec{e}_1|^{N + 2s}}dz + \int_{(B_1(0))^c}\frac{|z|^{\alpha} - 1 }{|z|^{\alpha}|z - \vec{e}_1|^{N + 2s}}dz\Big)
                  \\[1mm]
                  &\quad +  R^{N - \alpha}\Big(\int_{B_1(0)}\frac{|y|^{-\alpha} - 1}{|x - yR|^{N + 2s}}\,{\mathrm{d}}y + \int_{B_{{1 + \beta}}(0)\backslash{{ B_1(0)}}}\frac{|y|^{-\alpha} - {1}}{|x - yR|^{N + 2s}}\,{\mathrm{d}}y\Big)
                  \\[1mm]
                  &\ge C_{\alpha}\frac{f^{\beta}_{\alpha}(x)}{|x|^{2s}}
\end{align*}

When $x\in B_{2R}(0)\backslash B_R(0)$, by the construction of $\eta_{\beta}$ and the computation above, it holds
\begin{align*}
&\quad(-\Delta)^sf^{\beta}_{\alpha}(x)\\
                  & = \int_{\R^N}\frac{(|x|^{-\alpha} - |y|^{-\alpha}) + \eta_{\beta}(y)(|y|^{-\alpha} - |x|^{-\alpha})}{|x - y|^{N + 2s}}\,{\mathrm{d}}y
                   \\[1mm]
                  & \quad + \int_{\R^N}\frac{(\eta_{\beta}(x) - \eta_{\beta}(y))R^{-\alpha} + (\eta_{\beta}(y) - \eta_{\beta}(x))|x|^{-\alpha}}{|x - y|^{N + 2s}}\,{\mathrm{d}}y
                  \\[1mm]
                  &= (R^{-\alpha} - |x|^{-\alpha})\int_{\R^N}\frac{\eta_{\beta}(x)- \eta_{\beta}(y)}{|x - y|^{N + 2s}}\,{\mathrm{d}}y + \int_{\R^N}\frac{\eta_{\beta}(y)(|y|^{-\alpha} - |x|^{-\alpha})}{|x - y|^{N + 2s}}\,{\mathrm{d}}y
                  \\[1mm]
                  &= (R^{-\alpha} - |x|^{-\alpha})\frac{(-\Delta)^s\tilde{\eta}_{\beta}(R/x)}{R^{2s}} + \int_{\R^N}\frac{\eta_{\beta}(y)(|y|^{-\alpha} - |x|^{-\alpha})}{|x - y|^{N + 2s}}\,{\mathrm{d}}y
                  \\[1mm]
                  & :\ge {-C_{\beta}}{R^{-(\alpha + 2s)}} + K^{\beta}_{\alpha},
\end{align*}
where $C_{\beta}$ is a positive constant depending only on $\beta$.

For $K^{\beta}_{\alpha}$, by Change-Of-Variable Theorem and the non-increasing of $\tilde{\eta}_{\beta}$, it holds
\begin{align}\label{Ieq1}
|x|^{\alpha + 2s}K^{\beta}_{\alpha}
\nonumber& = \int_{B_{1}(0)}\frac{\eta_{\beta}(y|x|)(|y|^{-\alpha} - 1)}{|y - \vec{e}_1|^{N + 2s}}\,{\mathrm{d}}y + \int_{B^c_{1}(0)\cap B_{1 + \beta}(0)}\frac{\eta_{\beta}(y|x|)(|y|^{-\alpha} - 1)}{|y - \vec{e}_1|^{N + 2s}}\,{\mathrm{d}}y
\\[1mm]
\nonumber &\ge \int_{B_{1}(0)}\frac{\eta_{\beta}(|x|)(|y|^{-\alpha} - 1)}{|y - \vec{e}_1|^{N + 2s}}\,{\mathrm{d}}y + \int_{B^c_{1}(0)\cap B_{1 + \beta}(0)}\frac{\eta_{\beta}(|x|)(|y|^{-\alpha} - 1)}{|y - \vec{e}_1|^{N + 2s}}\,{\mathrm{d}}y
\\[1mm]
& \ge \eta_{\beta}(|x|)\int_{B_{3}(\vec{e}_1)}\frac{(|y|^{-\alpha} - 1)}{|y - \vec{e}_1|^{N + 2s}}\,{\mathrm{d}}y\\[1mm]
\nonumber&\ge \widetilde{C}_{\alpha},
\end{align}
where $\widetilde{C}_{\alpha}$ is a constant depending only on $\alpha$(it will hold $\widetilde{C}_{N - 2s} > 0$ if $\alpha = N - 2s$).

Finally, by the computation above, we conclude that if $R$ is large enough and $\epsilon$ is small enough, it holds
\begin{align}\label{AAeq4.5}
     \nonumber &\quad(-\Delta)^{s} f^{\beta}_{\alpha} + V_{\epsilon}(x)f^{\beta}_{\alpha} - \widetilde{P}_{\epsilon}(x)f^{\beta}_{\alpha}\\[2mm]
&\ge \left\{
       \begin{array}{ll}
         -(C_{\beta} + \widetilde{C}_{\alpha}) R^{-\alpha - 2s} + \big(\inf_{x\in\Lambda_{\epsilon}}V(x)\big)f^{\beta}_{\alpha}(x), & x\in B_{2R}(0)\backslash B_R(0), \\[5mm]
         C_{\alpha}\frac{f^{\beta}_{\alpha}}{|x|^{2s}} + V_{\epsilon}(x)f^{\beta}_{\alpha} - \widetilde{P}_{\epsilon}(x)f^{\beta}_{\alpha}, & x\in\R^N\backslash B_{2R}(0),
       \end{array}
     \right.\\[2mm]
\nonumber&\ge 0
\end{align}
for every $x\in \R^N\backslash B_R(0)$, where $\Lambda_{\epsilon} = \{x:\epsilon x + x_{\epsilon}\in\Lambda\}$. This completes the proof.


\end{proof}

\noindent At the last of this section, we give the proof of  Theorem \ref{th1.1}

\textbf{Proof of Theorem \ref{th1.1}}.
Let
\begin{align}\label{eq4.5}
\left\{
  \begin{array}{ll}
    \alpha \in\Big(\frac{2s}{p - 2},N - 2s\Big),\ \kappa = \frac{\alpha(p - 2) - 2s}{4},
\mathcal{P}_{\epsilon}(x) = \frac{\epsilon^{2s + 2\kappa}}{|x|^{2s + \kappa}}\chi_{\R^N\backslash\Lambda}(x), & \\[5mm]
    \overline{U}(x) = CR^{\alpha}f^{\beta}_{\alpha}(x), &
  \end{array}
\right.
\end{align}
where the function $f^{\beta}_{\alpha}$ is that in Proposition \ref{pr3.2}.

 It is easy to check that $\mathcal{P}_{\epsilon}$ satisfies the assumption \eqref{eq2.1}.

 By Proposition \ref{pr3.2}, letting the constant  $C>0$ above be large enough and $\tilde{v}_{\epsilon}(x) = \overline{U}(x) - v_{\epsilon}(x)$, we have
\begin{equation*}
  \left\{
    \begin{array}{ll}
      (-\Delta)^{s} \tilde{v}_{\epsilon}(x) + V_{\epsilon}(x)\tilde{v}_{\epsilon}(x) - \widetilde{P}_{\epsilon}(x)\tilde{v}_{\epsilon}(x)\geq 0, & \text{in}\ \mathbb{R}^N\backslash  B_{R}(0), \vspace{2mm}\\[5mm]
      \tilde{v}_{\epsilon}(x)\ge 0 & \text{in}\ B_R(0).
    \end{array}
  \right.
\end{equation*}
Then, since $\tilde{v}_{\epsilon}\in \dot{H}^s(\R^N)$   (when $\alpha$ is closed to $N - 2s$), testing the equation above against with $\tilde{v}^-_{\epsilon}(x)$, by the fractional Hardy inequality in \eqref{eq1.4}, we find $\tilde{v}^-_{\epsilon}(x) = 0,\ x\in\R^N$. Hence $\tilde{v}_{\epsilon}(x)\ge 0,\ x\in\R^N$. Especially, we have
$$
u_{\epsilon}(x)\le \frac{C\epsilon^{\alpha}}{\epsilon^{\alpha} + |x - x_{\epsilon}|^{\alpha}},\ \forall \,  x\in\R^N.
$$
Moreover, it holds
$$
\big(u_{\epsilon}(x)\big)^{p - 2}\le \mathcal{P}_{\epsilon}(x),\ \forall  \, x\in\R^N\backslash\Lambda.
$$

As a result, $u_{\epsilon}$ solves the origin problem. This completes the proof.

\vspace{0.5cm}
\section{Further results}\label{s5}

In this section, we will consider \eqref{eq1.1} with general nonlinearity, i.e.,
\begin{equation}\label{eq5.1}
  \epsilon^{2s}(-\Delta)^su + V(x)u = f(u),
\end{equation}
where the potential $V(x)$ is the same as before, the nonlinearity $f:\mathbb{R}\to \mathbb{R}$ is assumed to satisfy the following properties:

$(\textbf{f}_1)$ $f$ is an odd function and $f(t) = o(t^{1 + \tilde{\kappa}})$ as $t\to 0^+$, where $\tilde{\kappa} = \frac{2s + 2\kappa}{\alpha -\nu}>0$ with $\nu>0$ is a small parameter and $\kappa$ is the parameter in \eqref{eq4.5}.

$(\bf{f}_2)$  $\lim_{t\to\infty}\frac{f(t)}{t^p} = 0$ for some $1<p<2^*_{s} - 1$.

$({ \bf{f} }_3)$ There exists $2<\theta \leq p + 1$ such that
$$
0\leq\theta F(t) < f(t)t\ \text{for all}\ t>0,
$$
\ \ \ \ \ \ \ \ \ where $F(t) = \ds\int_{0}^tf(\alpha)\,{\mathrm{d}} \alpha$.

$(\textbf{f}_4)$ The map $t\mapsto \frac{f(t)}{t}$ is increasing on $(0,+\infty)$.

\vspace{0.2cm}

\noindent We have the following results  which are  same as Theorem \ref{th1.1}.

\begin{theorem}\label{th5.1}
Let $N > 2s$, $s\in(0,1)$, $p\in\Big(2+\frac{2s}{N - 2s},2^*_s\Big)$, $f$ satisfy $(\bf{f}_1)$- $(\bf{f}_4)$ and $V$ be the same as before. Then problem \eqref{eq1.1} has a positive solution $\hat{u}_{\epsilon}\in \mathcal{D}^s_{V,\epsilon}(\mathbb{R}^N)$ if $\epsilon>0$ is small enough. Moreover, there exist a $\hat{x}_{\epsilon}\in\Lambda$ and an $\alpha\in(2s/(p - 2),2^*_s)$ such that
$$
 \lim_{\epsilon\to 0}V(\hat{x}_{\epsilon}) = \min_{\Lambda}V(x)
 $$
 and
\begin{align*}
u_{\epsilon}(x)\le \frac{C\epsilon^{\alpha}}{\epsilon^{\alpha} + |x - \hat{x}_{\epsilon}|^{\alpha}},
\end{align*}
where $C$ is positive constant.

\end{theorem}

\begin{proof}

We define the penalized nonlinearity   $\hat{g}_{\epsilon}:\mathbb{R}^N\times\mathbb{R}$ as
$$
\hat{g}_{\epsilon}(x,s):=\chi_{\Lambda}f(s_+) + \chi_{\mathbb{R}^N\backslash\Lambda}\min\{f(s_+),\mathcal{P}_{\epsilon}(x)s_+\}.
$$
In the sequel, we denote $\widehat{G}_{\epsilon}(x,t) = \ds\int_{0}^{t}\hat{g}_{\epsilon}(x,s) \,{\mathrm{d}}s $ and define the penalized superposition operators $\hat{\mathbf{g}}_{\epsilon}$ and $\widehat{\mathbf{G}}_{\epsilon}$ as
$$
\hat{\mathbf{g}}_{\epsilon}(u)(x) = \hat{g}_{\epsilon}(x,u(x))\ \text{and}\ \widehat{\mathbf{G}}_{\epsilon}(u)(x) = \widehat{G}_{\epsilon}(x,u(x)).
$$
Accordingly, the penalized functional $\hat{J}_{\epsilon}:\mathcal{D}^s_{V,\epsilon}(\mathbb{R}^N)\to \mathbb{R}$ is given by
$$
\hat{J}_{\epsilon}(u) = \frac{1}{2}\int_{\mathbb{R}^N}(\epsilon^{2s}|(-\Delta)^{s/2} u|^2 + V(x)|u|^2) \,{\mathrm{d}}y   - \int_{\mathbb{R}^N}\widehat{\mathbf{G}}_{\epsilon}(u) \,{\mathrm{d}}y.
$$
By conditions $(f_2)$ and $(f_3)$, we can verify using the same argument in the proof of Lemma \ref{le2.2} that $\hat{J}_{\epsilon}$ is $C^1$ and satisfies (P.S.) condition. Then similar to Lemma \ref{le2.3}, one can show that the Mountain Pass value
$$
\hat{c}_{\epsilon} := \inf_{\gamma\in\widehat{\Gamma}_{\epsilon}}\max_{t\in [0,1]}\hat{J}_{\epsilon}(\gamma(t))
$$
can be achieved by a positive function $\hat{u}_{\epsilon}\in \mathcal{D}^s_{V,\epsilon}$, where
\begin{equation*}
\widehat{\Gamma}_{\epsilon} : = \{\gamma\in (C[0,1],\mathcal{D}^s_{V,\epsilon}(\R^N)):\gamma(0) = 0,\ \hat{J}_{\epsilon}(\gamma(1)) < 0\}.
\end{equation*}
Furthermore, it holds
\begin{equation}\label{eq5.2}
\epsilon^{2s}(-\Delta)^s\hat{u}_{\epsilon} + V(x)\hat{u}_{\epsilon} = \hat{\mathbf{g}}_{\epsilon}(\hat{u}_{\epsilon}).
\end{equation}

From \cite{S,FQT} and the condition on $f$, the following limiting problem
$$
(-\Delta)^su + au = f(u)
$$
has a positive ground state $\hat{u}$. Moreover, the least energy
$$
\hat{\mathcal{C}}(a) = \frac{1}{2}\Big(\int_{\R^N}|(-\Delta)^{s/2}\hat{u}|^2 + a|\hat{u}|^2\Big) \,{\mathrm{d}}x   - \int_{\R^N}F(\hat{u}) \,{\mathrm{d}}x
$$
is continuous and increasing.  Then, by the same argument in Lemma \ref{le2.5}, there exists a $\hat{x}_{\epsilon}\in\Lambda$ such that
\begin{eqnarray}\label{eq5.9}
\begin{split}
  &(1')\ \liminf_{\epsilon\to 0}\|\hat{u}_{\epsilon}\|_{L^{\infty}(B_{\epsilon\rho}(\hat{x}_{\epsilon}))}> 0,\\
  &(2')\ \lim_{\epsilon\to 0}V(\hat{x}_{\epsilon}) = \inf_{\Lambda}V,\\
  &(3')\ \lim\limits_{{R\to \infty}\atop{{\epsilon\to 0}}}\|\hat{u}_{\epsilon}\|_{L^{\infty}(U\backslash B_{\epsilon R}(\hat{x}_{\epsilon}))} =  0.
\end{split}
\end{eqnarray}
As a result, using $(f_1)$ to linearize \eqref{eq5.2}, we get
\begin{equation}\label{eq5.10}
  \left\{
    \begin{array}{ll}
      \epsilon^{2s}(-\Delta)^{s} \hat{u}_{\epsilon} + (1 - \delta)V\hat{u}_{\epsilon}\leq {P}_{\epsilon}\hat{u}_{\epsilon}, & \text{in}\ \mathbb{R}^N\backslash  B_{R\epsilon}(x_{\epsilon}), \vspace{2mm}\\
      \hat{u}_{\epsilon}\leq \widehat{C}_{\infty} & \text{in}\ \Lambda.
    \end{array}
  \right.
\end{equation}
Then, by the same argument in Section \ref{s4}, it holds
$$
\hat{u}_{\epsilon}(x)\le \frac{C\epsilon^{\alpha}}{\epsilon^{\alpha} + |x - \hat{x}_{\epsilon}|^{\alpha}},\ \forall x\in\R^N,
$$
where $\alpha$ is the same as that in \eqref{eq4.5}. Following, for every $x\in\R^N\backslash\Lambda$, since $\alpha\tilde{\kappa}>2s + 2\kappa$, letting $\epsilon>0$ be small enough, it holds
\begin{align*}
 \frac{f(\hat{u}_{\epsilon})}{\hat{u}_{\epsilon}}&\le (\hat{u}_{\epsilon})^{\tilde{\kappa}}\le \frac{C\epsilon^{\alpha\tilde{\kappa}}}{|x|^{\alpha\tilde{\kappa}}}\le \frac{\epsilon^{2s + 2\kappa}}{|x|^{2s + \kappa}} = \mathcal{P}_{\epsilon}(x).
\end{align*}

As a result, $\hat{u}_{\epsilon}$ is a solution to the origin problem \eqref{eq5.1}. This completes the paper.

\end{proof}

\bigskip
\bigskip
{\bf Acknowledgements: }   L.  Duan  is partially  supported by China Scholarship Council
(No.201906770024).
Y.  Peng  is supported by  the grant NSFC (No. 11501143).


\begin{thebibliography}{99}


\bibitem{100}
C. Alves, O. Miyagaki, Existence and concentration of solution for a class of fractional elliptic equation in $\mathbb{R}^N$ via penalization method. Calc. Var. Partial Differential Equations \textbf{55}, 1-19 (2016)

\bibitem{5}
{A. Ambrosetti, M. Badiale,S. Cingolani, }{Semiclassical states of nonlinear Schr\"{o}dinger equations}. Arch. Ration. Mech. Anal.
\textbf{140}(3), 285-300 (1997)

\bibitem{6}
A. Ambrosetti, A.  Malchiodi, {Perturbation methods and semilinear elliptic problems on $\mathbb{R}^N$}. Progress in Mathmatics, vol. 240. Birfh$\ddot{a}$user Verlag, Basel (2006)

\bibitem{14}
A. Ambrosetti, A. Malchiodi, W. M. Ni, {Singularly perturbed elliptic equations with symmetry: existence
of solutions concentrating on spheres}. I. Comm. Math. Phys. \textbf{235}(3), 427-466 (2003)

\bibitem{AA}
A. Ambrosetti, A. Malchiodi, Concentration phenomena for NLS: Recent results and new
perspectives, perspectives in nonlinear partial differential equations. Contemp. Math.
446(2007), 19-30.


\bibitem{V.Ambrosio-2019}
V. Ambrosio,   Concentrating solutions for a class of nonlinear fractional Schr\"odinger equations in $\R^N$ , Rev. Mat. Iberoam. \textbf{35 }(2019), no. 5, 1367-
1414



\bibitem{APX}
X. An, S. Peng, C. Xie, {Semi-classical solutions for fractional Schr\"odinger equations with potential vanishing at infinity}. J. Math. Phys. \textbf{60}, 021501 (2019).




\bibitem{30}
{D. Bonheure, J.  Van Schaftingen, } {Groundstates for the nonlinear Schr\"{o}dinger equation with potential
vanishing at infinity}. Ann. Mat. Pura Appl. (4) \textbf{189}(2), 273-301 (2010)

\bibitem{BaNa}
N. Ba,  D.  Deng., S. Peng,  Multi-peak bound states for Schr\"odinger equations with compactly
supported or unbounded potentials. Ann. Inst. H. Poincare Anal. Non Lineaire
\textbf{27} (2010), 1205-1226.



\bibitem{Byeon-Wang}
J. Byeon,  Z.-Q. Wang,  Spherical semiclassical states of a critical frequency for Schr\"odinger
equations with decaying potentials. J. Eur. Math. Soc. (JEMS) \textbf{8 }(2006), 217-228.


\bibitem{Cao-Peng}
D. Cao,  S.  Peng,  Semi-classical bound states for Schr\"odinger equations with potentials
vanishing or unbounded at infinity. Comm. Partial Differential Equations \textbf{34} (2009), 1566-1591.

\bibitem{25}
G. Chen,  Y. Zheng, {Concentration phenomena for fractional noninear Schr\"{o}dinger equations.} Commun. Pure Appl. Anal. \textbf{13}, 2359-2376 (2014)


\bibitem{24}
L. Caffarelli, L. Silvestre,  { An extension problem related to the fractional Laplacian}. Commun. Partial Differ. Equ. \textbf{32}, 1245-1260 (2007)


\bibitem{7}
{S. Cingolani,M.  Lazzo, }  {Multiple semiclassical standing waves for a class of nonlinear Schr\"{o}dinger
equations.} Topol. Methods Nonlinear Anal. \textbf{10}(1), 1-13 (1997)

\bibitem{8}
S. Cingolani, M. Lazzo, {Multiple positive solutions to nonlinear Schr\"{o}dinger equations with competing
potential functions}. J. Differential Equations \textbf{160}(1), 118-138 (2000)



\bibitem{J.Davila-M.del Pino-J.Wei-2014}
J. D$\acute{a}$vila, M. del Pino, J. Wei, Concentrating standing waves for the fractional nonlinear Schr\"odinger equation, J. Differential Equations \textbf{256} (2014),
no. 2, 858-892


\bibitem{DDW}
J. D$\acute{a}$vila, M. del Pino, J. Wei, Nonlocal $s$-minimal surfaces and Lawson cones. J. Differential Geom. \textbf{109}(1), 111-175 (2018).

\bibitem{DF}
B. Dyda,  R.L. Frank, Fractional Hardy-Sobolev-Mazya inequality for domains. Studia Math. \textbf{208}(2), 151-166 (2012).


\bibitem{9}
M. del Pino, P.L. Felmer,  {Local Mountain Passes for semilinear elliptic problems in unbounded domains.}
Calc. Var. Partial Differential Equations \textbf{4}(2), 121-137 (1996)

\bibitem{10}
M. del Pino, P.L. Felmer,  {Semi-classical states for nonlinear Schr\"{o}dinger equations}. J. Funct. Anal.
\textbf{149}(1), 245-265 (1997)

\bibitem{MM3}
M. del Pino, P.L.  Felmer, {Multi-peak bound states for nonlinear Schr\"{o}dinger equations}. Ann. Inst. H. Poincar\`{e}, Analyse non lin\`{e}aire \textbf{15}, 127-149 (1998)


\bibitem{15}
M. del Pino, M. Kowalczyk, J. Wei,  {Concentration on curves for nonlinear Schr\"{o}dinger equations}.
Comm. Pure Appl. Math. \textbf{60}(1), 113-146 (2007)




\bibitem{4}
E. Di Nezza,  G. Palatucci, E. Valdinoci, {Hitchhikers guide to the fractional Sobolev spaces}. Bull. Sci. Math. \textbf{136} 521-573  (2012)






\bibitem{19}
L. Frank, E. Lenzmann, {Uniqueness of non-linear ground states for fractional Laplacians in $\mathbb{R}$}. Acta. Math. \textbf{210}, 261-318, (2013)



\bibitem{20}
L. Frank, E. Lenzmann,  L. Silvestre,  {Uniqueness of radial solutions for the fractional Laplacians}. Comm. Pure. Appl. Math. \textbf{69}  1671-1726(2016)



\bibitem{FS}
L. Frank, R. Seiringer,  Non-linear ground state representations
and sharp Hardy inequalities. J. Funct. Anal. \textbf{255} 3407-3430(2008)

\bibitem{FQT}
 P. Felmer, A. Quaas, J. Tan,  Positive solutions of the nonlinear Schr\"odinger equation with the fractional Laplacian. Proc. Roy. Soc. Edinburgh Sect. A \textbf{142}, 1237-1262 (2012)


\bibitem{26}
M. M. Fall, F. Mahmoudi, E. Valdinoci, {Ground states and concentration phenomena for the fractional Schr\"{o}dinger equation}. Nonlinearity
\textbf{28}, 1937-1961 (2015)

\bibitem{2}
N. Laskin, {Fractional Schr\"{o}dinger equation}. Phys. Lett. A \textbf{268} 298-305 (2000)

\bibitem{3}
N. Laskin, {Fractional quantum mechanics and Levy path integrals}. Phys. Lett. A \textbf{268} 298-305 (2000)




\bibitem{12}
Y. G. Oh,  {Existence of semiclassical bound states of nonlinear Schr\"{o}dinger equations with potentials of
the class $(V)_a$}. Comm. Partial Differential Equations. \textbf{13}(12), 1499-1519 (1988)





\bibitem{13}
P. H. Rabinowitz, {On a class of nonlinear Schr\"{o}dinger equations}. Z. Angew. Math. Phys. \textbf{43}(2), 270-291
(1992)


\bibitem{S}
S. Secchi,  Ground state solutions for nonlinear fractional Schr\"odinger equations in $\mathbb{R}^N$. J. Math. Phys. \textbf{54}, 031501 (2013)

\bibitem{27}
X. Shang, J. Zhang, {Concentrating solutions of nonlinear fractional Schr\"{o}dinger equation with potentials}. J. Differential Equations  \textbf{258}, 1106-1128 (2015)

\bibitem{MV}
V. Moroz, J. Van Schaftingen, Semiclassical stationary states for nonlinear Schr\"odinger equations with fast decaying potentials, Calc. Var. Partial Differential Equations \textbf{37} (2010), 1-27

\bibitem{28}
{M. Vitaly, J. Van Schaftingen, }  {Semi-classical states for the Choquard equation}. Calc. Var. Partial Differential Equations \textbf{52}(1-2) (2015)




\bibitem{MW}
    \newblock M. Willem,
    \newblock \emph{Minimax Theorems},
    \newblock Birkh\"{a}user, Boston, 1996.



%






\end{thebibliography}
\end{document}